\documentclass[11pt]{amsart}

\usepackage{amsmath,amsfonts,amscd,amssymb}

\theoremstyle{definition}
\newtheorem*{Theorem}{Theorem 1}
\newtheorem*{acknowledgements}{Acknowledgements}

\def\injects{\lhook\joinrel\rightarrow}
\let\surjects\twoheadrightarrow
\let\iso\cong
\let\normal\triangleleft
\def\<{\raise0.5pt\hbox{$\,\scriptstyle<\,$}}
\def\Q{{\mathbb Q}}
\def\Z{{\mathbb Z}}
\DeclareMathOperator\Aut{Aut}
\DeclareMathOperator\Ind{Ind}
\def\smlneq{\raise1.5pt\hbox{$\,\scriptstyle\lneq\hskip 2.2pt$}}
\def\triv{{\mathbf 1}}

\begin{document}

\llap{.\hskip 10cm} \vskip -10mm

\title{Solomon's induction in quasi-elementary groups}
\author{Tim Dokchitser}
\begin{abstract}
Given a finite group $G$, we address the following question:
which multiples of the trivial representation
are linear combinations of inductions of
trivial representations from proper subgroups of $G$?
By Solomon's induction theorem, all multiples are
if $G$ is not quasi-elementary.
We complement this by showing that all multiples of $p$ are if $G$ is
$p$-quasi-elementary and not cyclic, and that this is best possible.
\end{abstract}
\maketitle

A finite group $G$ is ($p$-){\em quasi-elementary\/} if it has a cyclic
normal \hbox{subgroup} of $p$-power index.
Solomon's induction theorem
(\cite{Sol} Thm. 1 with ${\mathbf K}=\Q$ or \cite{Isa}~Thm.~8.10) asserts
that the trivial character of
any finite group $G$ is a linear combination of inductions
$$
  \triv_G = \sum_H n_H \Ind_{H}^G\triv_{H},
$$
for some $n_H\in\Z$ and quasi-elementary subgroups $H\<G$ (possibly
with {respect} to different primes.)
In particular, $\triv_G$ is a linear combination of inductions
of $\triv_H$ from proper subgroups $H\<G$ when $G$ is not quasi-elementary.
In this note we show that
this statement is never true
for $\triv_G$ when $G$ is $p$-quasi-elementary, but is always true for $p\triv_G$,
unless $G$ is cyclic.
Both claims are easy to prove, but they do not appear to be in print.

We call a formal linear combination $\sum_H n_H H$ of
(not necessarily proper)
subgroups of $G$
a {\em Brauer relation\/} in $G$ 
if $\sum_H n_H\Ind_H^G\triv_H=0$.

\begin{Theorem}
Let $G$ be a finite group, and let $I\subset \Z$ be the set of integers
that can occur as $n_G$ in Brauer relations $\sum_H n_H H$. Then
\begin{itemize}
\item $I=\{0\}$ if $G$ is cyclic,
\item $I=p\Z$ if $G$ is $p$-quasi-elementary and not cyclic, and
\item $I=\Z$ if $G$ is not quasi-elementary.
\end{itemize}
\end{Theorem}

\begin{proof}
Clearly $I$ is an ideal in $\Z$.
It is easy to see that cyclic groups have no non-zero Brauer relations,
whence the first claim. For the last claim, Solomon's induction theorem
shows that $1\in I$ for non-quasi-elementary groups.
Assume from now on that $G$ is $p$-quasi-elementary and not cyclic.
It remains to show that

\begin{itemize}
\item[a)] $p\,\triv_G$ is in $\Z$-span of $\Ind_H^G\triv_H$ for $H\smlneq G$, and
\item[b)] $\triv_G$ is not.
\end{itemize}

Let $C\normal G$ be a cyclic subgroup of $p$-power index.
The elements of $C$ of order prime to $p$ form a cyclic subgroup $C'$
which is characteristic in $C$ and therefore normal in $G$.
Replacing $C$ by $C'$, we may assume that $p\nmid |C|$. Now $G=C\rtimes P$
by the Schur-Zassenhaus theorem, with $P<G$ its $p$-Sylow.

a) We proceed by induction on $|G|$.


If $N\normal G$ is a normal subgroup and $\phi: G\surjects Q=G/N$ the quotient map,
then any Brauer relation $\sum_U n_U U\>\> (U\<Q)$ in $Q$ lifts
to a relation $\sum_U n_U \phi^{-1}(U)$ in $G$.
Also note that $Q$ is $p$-quasi-elementary as well.
Thus if there exists an $N$ with $G/N$ {\em non-cyclic},
we may apply the theorem to $G/N$ (by induction) and lift the resulting
relation back to $G$. Hence assume that there is no such $N$. This implies that

\begin{itemize}
\item {\em $P$ is cyclic.} Otherwise, let
$N=C\rtimes\text{(Frattini subgroup of $P$)}$. Then
$G/N\iso(C_p)^n$ for some $n>1$, which is not cyclic.
\item {\em The action of $P$ on $C$ is non-trivial.} Otherwise $G$ is cyclic.
\item {\em The action of $P$ on $C$ is faithful.} Otherwise $G$ modulo
the kernel of this action is a non-cyclic quotient.
\item {\em $C$ has prime power order}.
Otherwise $C=U_1\times U_2$ with non-trivial $U_1$, $U_2$, and either
$G/U_1$ or $G/U_2$ is a non-cyclic quotient.
\end{itemize}
In particular, because $P$ and $C$ have coprime order and $P\injects\Aut C$,
the order of $C$ cannot be a power of 2.
\begin{itemize}
\item {\em $C$ has prime order}. Otherwise take $U=C_{l^{k-1}}<C_{l^k}=C$. Then
$$
  (\Z/l\Z)^\times \times (\Z/l^{k-1}\Z) \iso
  \Aut(C)\to\Aut(C/U)\iso (\Z/l\Z)^\times
$$
is bijective on elements of order prime to $l$,
so $G$ acts faithfully on $C/U$, and $G/U$ is a non-cyclic quotient.
\end{itemize}
Finally, now $G=C_l\rtimes C_{p^k}$ with faithful action, and it is easy
to check that 
$$
  C_{p^{k-1}}-p\> C_{p^k} - C_l\rtimes C_{p^{k-1}} + p\> G =0
$$
is a Brauer relation. 

b) Let $R=\sum n_H H$ be a Brauer relation.
Restricting each term $\Ind_H^G\triv_H$ to $C$ using Mackey's decomposition,
we find a Brauer relation in $C$, namely
$$
  \sum n_H [G:HC]\> (H \cap C).
$$
As cyclic groups have no non-trivial relations,
all terms, in particular the ones with $C$ must cancel.
These come from subgroups $H\supset C$, that is the ones of the form
$H=C\rtimes P_H$ with $P_H\subset P$. Thus,
$$
  \smash{\sum_{H\supset C}} n_{H} [P:P_H] = 0.
$$
\vskip 3mm
All terms except the one with $P_H=P$ (i.e. $H=G$) are divisible by $p$,
so $n_G$ must be a multiple of $p$. This shows that $n_G\ne 1$.
\end{proof}

\begin{acknowledgements}
This research is supported by a Royal Society University Research fellowship.
The author would like to thank Alex Bartel, Vladimir Dokchitser
and the referee for
helpful comments.
\end{acknowledgements}


\begin{thebibliography}{9}

\bibitem{Isa}
I. M. Isaacs, Character theory of finite groups, AMS Chelsea Publishing, 2006.

\bibitem{Sol}
L. Solomon, The representation of finite groups in algebraic number fields,
J. Math. Soc. Japan 13 no. 2 (1961), 144--164.

\end{thebibliography}
\end{document}